\documentclass[dvipdfmx]{compositio}

\usepackage{amsmath,amssymb,amsthm} 
\usepackage[all]{xy}
\usepackage{extarrows}
\usepackage{graphicx}
\usepackage{bm} 
\usepackage{comment, url}
\usepackage[normalem]{ulem}
\usepackage{multicol}

\usepackage[pagewise, mathlines]{lineno} 
\usepackage{time}

\usepackage{etoolbox}          

\newcommand*\linenomathpatch[1]{%
  \cspreto{#1}{\linenomath}%
  \cspreto{#1*}{\linenomath}%
  \csappto{end#1}{\endlinenomath}%
  \csappto{end#1*}{\endlinenomath}%
}
\newcommand*\linenomathpatchAMS[1]{%
  \cspreto{#1}{\linenomathAMS}%
  \cspreto{#1*}{\linenomathAMS}%
  \csappto{end#1}{\endlinenomath}%
  \csappto{end#1*}{\endlinenomath}%
}

\expandafter\ifx\linenomath\linenomathWithnumbers 
 \let\linenomathAMS\linenomathWithnumbers
 \patchcmd\linenomathAMS{\advance\postdisplaypenalty\linenopenalty}{}{}{}
\else
  \let\linenomathAMS\linenomathNonumbers
\fi

\linenomathpatch{equation}
\linenomathpatchAMS{gather}
\linenomathpatchAMS{multline}
\linenomathpatchAMS{align}
\linenomathpatchAMS{alignat}
\linenomathpatchAMS{flalign}

\makeatletter
\patchcmd{\mmeasure@}{\measuring@true}{
  \measuring@true
  \ifnum-\linenopenaltypar>\interdisplaylinepenalty
    \advance\interdisplaylinepenalty-\linenopenalty
  \fi
  }{}{}
\makeatother

\newtheorem{thm}{Theorem}[section]
\newtheorem{lem}[thm]{Lemma}

\newtheorem{prop}[thm]{Proposition}  

\theoremstyle{remark}

\theoremstyle{definition}

\newtheorem{rem}[thm]{Remark} 

\newtheorem{eg}[thm]{Examples}       
                 
\newtheorem{conj}[thm]{Conjecture}

\newtheorem{def/prop}[thm]{Definition/Proposition}

\numberwithin{equation}{section}

\usepackage[usenames]{color}

\def\tr{\mathop{\mathrm{tr}}\nolimits}
\def\det{\mathop{\mathrm{det}}\nolimits}

\def\SL{\mathop{\mathrm{SL}}\nolimits}
\def\GL{\mathop{\mathrm{GL}}\nolimits}

\def\mod{\mathop{\mathrm{mod}}\nolimits}

\newcommand{\mf}[1]{{\mathfrak{#1}}}

\newcommand{\bb}[1]{{\mathbb{#1}}}
\newcommand{\mca}[1]{{\mathcal{#1}}}

\newcommand{\To}{\longrightarrow}

\newcommand{\inj}{\hookrightarrow}
\newcommand{\surj}{\twoheadrightarrow}

\newcommand{\congto}{\overset{\cong}{\to}}

\newcommand{\Z}{\bb{Z}}
\newcommand{\Zp}{\bb{Z}_{p}}
\newcommand{\Q}{\bb{Q}}
\newcommand{\Qp}{\bb{Q}_{p}}
\newcommand{\R}{\bb{R}}
\newcommand{\C}{\bb{C}}

\newcommand{\F}{\bb{F}}
\newcommand{\p}{\mf{p}}
\renewcommand{\P}{\mf{P}}

\newcommand{\mh}{\textsc{m}}

\newcommand{\ol}{\overline}
\newcommand{\ul}{\underline}

\newcommand{\ds}{\displaystyle}

\newcommand{\wt}[1]{{\widetilde{#1}}}
\newcommand{\wh}[1]{{\widehat{#1}}}

\DeclareMathOperator*{\restprod}%
 {\mathchoice{\ooalign{\ensuremath{\displaystyle\prod}\crcr\ensuremath{\displaystyle\coprod}}}%
             {\ooalign{\ensuremath{\textstyle\prod}\crcr\ensuremath{\textstyle\coprod}}}%
             {\ooalign{\ensuremath{\scriptstyle\prod}\crcr\ensuremath{\scriptstyle\coprod}}}%
             {\ooalign{\ensuremath{\scriptscriptstyle\prod}\crcr\ensuremath{\scriptscriptstyle\coprod}}}%
 }

\newcommand{\pmx}[1]{\begin{pmatrix}#1\end{pmatrix}}
\newcommand{\spmx}[1]{{\small \pmx{#1}}}

\def\be { \begin{equation} }
\def\ee { \end{equation} }

\title[Twisted Iwasawa invariants of knots] 
{Twisted Iwasawa invariants of knots}

\author{Ryoto Tange} 
\email{rtange.math@gmail.com}
\address{Department of Mathematics, School of Education, Waseda University; 
1-104 Totsuka-cho, Shinjuku-ku, 169-8050, Tokyo, Japan} 

\author{Jun Ueki} 
\email{uekijun46@gmail.com} 
\address{Department of Mathematics, Faculty of Science, Ochanomizu University; 
2-1-1 Otsuka, Bunkyo-ku, 112-8610, Tokyo, Japan} 

\keywords{knot, genus, fiberedness, twisted Alexander polynomial, Iwasawa invariants, profinite rigidity, arithmetic topology 
\hspace{10cm} \framebox{
\today \ \now}
} 

\begin{document}
	
\begin{abstract} 
Let $p$ be a prime number and $m$ an integer coprime to $p$. 
In the spirit of arithmetic topology, we introduce the notions of the twisted Iwasawa invariants $\lambda, \mu, \nu$ of ${\rm GL}_N $-representations and $\Z/m\Z\times \Zp$-covers of knots.
We prove among other things that the set of Iwasawa invariants determines the genus and the fiberedness of a knot, yielding their profinite rigidity. Several intuitive examples are attached. 
We further prove the $\mu=0$ theorem for ${\rm SL}_2$-representations of twist knot groups and give some remarks. 
\end{abstract}


\maketitle 

\setcounter{tocdepth}{2} 
{\small 
\tableofcontents } 


\section{Introduction} 
The analogy between the Alexander--Fox theory and the Iwasawa theory for infinite cyclic covers is historically important
amongst the analogies between knots and 
non-zero prime ideals, or 3-manifolds and the rings of integers of algebraic number fields in arithmetic topology \cite{Mazur1963, Morishita2012},   
whereas the twisted Alexander--Fox theory of knots and the Iwasawa theory associated to Galois representations have been seen as the second generation of those classics. 
Besides, $p$ being a prime number, $p$-adic refinements of the Alexander--Fox theory for knots and links, which have been developed by the authors and others, recently turned out to be useful in the study of profinite rigidity. 
In this article, we will 
establish a $p$-adic refinement of the twisted Alexander--Fox theory of knot group representations, 
aiming to facilitate the studies in both sides. 

Let $\Delta_K(t)$ denote the classical Alexander polynomial of a knot $K$ in $S^3$. 
Then Fox's formula asserts for each $n\in \Z_{>0}$ 
that the size of the torsion subgroup of $H_1(X_n)$ of the $\Z/n\Z$-covers \[X_n\to X=S^3-K\] is given by the $n$-th cyclic resultant of $\Delta_K(t)$ (cf.~\cite{Weber1979}), 
and the Mahler measure $\mh(\Delta_K(t))$ describes their asymptotic behavior \cite{GS1991}. 
A $p$-adic analogue of the asymptotic formula is given by the Gauss norm $\mh_p(\Delta_K(t))$ \cite{Ueki4} and its refinement is known as the Iwasawa type formula \cite{HillmanMateiMorishita2006, KadokamiMizusawa2008, KadokamiMizusawa2013, Ueki2, Ueki3, Ueki-survey1}. 
In our previous articles \cite{TangeRyoto2018JKTR, Ueki6, Ueki10}, we established Fox's formula and the asymptotic formula using of the Mahler measures for ${\rm GL}_N$-representations of knot groups. 
In this article, we introduce the notion of \emph{the twisted Iwasawa invariants} of knot group representations and investigate them.

The twisted Iwasawa invariants of knots and links are easily calculated from the twisted Alexander polynomials, so our topological interest will focus rather on if such weak invariants convey important information. 
In fact, we derive some results on the profinite rigidity, that has been of high interest in recent days (cf.\cite{BoileauFriedl2020AMS, BridsonReid2020AMS, Reid-ICM2018, BridsonMcReynoldsReidSpitler2020, Wilkes2019Israel, YiLiu2023Duke, YiLiu2023Invent, YiLiu2023Peking}). 

Let $\Zp$ denote the ring of $p$-adic integers defined by $\Zp=\varprojlim_r \Z/p^r\Z$, where $r$ runs through positive integers. An inverse system of $\Z/p^r\Z$-covers is called a $\Zp$-cover. 
Let $F$ be a number field, $S$ a finite set of maximal ideals of the ring $O_F$ of integers, and let $O=O_{F,S}$ denote the ring of $S$-integers, which is the smallest Dedekind subdomain of $F$ containing $O_F$ and inverse elements of $S$. 
For each $i$-th twisted Alexander module associated to a knot group representation $\rho:\pi_K\to {\rm GL}_NO$ with $N\in \Z_{>0}$ 
such that $H_i(X_\infty,\rho)$ is a torsion $O[t^\Z]$-module, where $X_\infty\to X$ denote the $\Z$-cover, 
we have a generalization of Fox's formula and the asymptotic formulas using of the Mahler measures (Theorem \ref{thm.Fox}), as well as the Iwasawa type formula; For each integer $m$ coprime to $p$, 
there exists some $\lambda,\mu,\nu \in \Z$ such that for any integer $r\gg 0$, the equality 
\[||H_i(X_{mp^r},\rho)_{\rm tor}||_p=p^{-(\lambda r+\mu p^r +\nu)}\]
holds (Theorem \ref{thm.Iwasawa}), where the left-hand side is the inverse of the size of the $p$-torsion subgroup of the twisted homology group. These integers $\lambda,\mu,\nu$ are called \emph{the twisted Iwasawa invariants} of the $\Z/m\Z\times \Zp$-cover and $\rho$. 

We remark that for each maximal ideal $\p$ of $O$ satisfying $\p\cap \Z=(p)$, 
the Iwasawa module \[\mca{A}_\p=\varprojlim_r H_i(X_{mp^r},\rho\otimes O_\p)\] of the $\Z/m\Z\times \Zp$-cover decomposes into the direct sum of the modules associated to characters of $\Z/m\Z$. Such a decomposition played an important role in the proof of the Iwasawa main conjecture for cyclotomic $\Zp$-extensions (cf.\cite[pp.291--292]{Washington}). 

In number theory, the Iwasawa $\lambda$-invariants of $\Zp$-fields are analogues of the genera of Riemann surfaces. 
Indeed, the analogy between the Riemann--Hurwitz formula of branched covers of Riemann surfaces and Kida's formula for $p$-power extensions (i.e., finite $p$-extensions) of $\Zp$-fields is pointed out in \cite{Kida1980, Iwasawa1981, Gras1978}. 
In addition, Kida's formula for elliptic curves is given by Hachimori--Matsuno \cite{HachimoriMatsuno1999} 
and that for $p$-power covers of $\Zp$-covers of links is formulated in \cite{Ueki2}. 
Thus, it might be a natural question to ask if there is any relation between the Iwasawa $\lambda$-invariants and the genera of knots. 

In this article, we prove that the set of the twisted Iwasawa invariants of $\Z/m\Z\times \Zp$-covers determine the degree and the monicness of the twisted Alexander polynomial (Theorems \ref{thm.degree}, \ref{thm.monic}), 
invoking some elementary $p$-adic number theory used in \cite{Ueki4, Ueki5, Ueki6}. 
Then deep results due to Friedl--Vidussi and others in \cite{FriedlVidussi2011AnnMath, FriedlVidussi2013JEMS, FriedlVidussi2015Crelle}+\cite{FriedlKim2008,FriedlNagel2015Illinois} yield that the sets of $\lambda$'s and $\mu$'s determine the genus and the fiberedness of any knot respectively in several certain senses (Theorem \ref{thm.fib.genus}). 
Since our results also hold for any $\p$-adic representations, Boileau--Friedl's result on the profinite rigidity of the genus and the fiberedness \cite[Theorem 1.2]{BoileauFriedl2020AMS} may be refined via the Iwasawa invariants (Theorem \ref{thm.profinite}). 

The Iwasawa $\mu$-invariant is a $p$-adic analogue of the Mahler measure, so it may be seen as a dynamical invariant, and determining whether the equality $\mu=0$ holds for cyclotomic $\Zp$-extensions has been a leading conjecture or has lead to important theorems (cf.~\cite{Iwasawa1973, FerreroWashington1979}). 
For $\Z$-covers of links, we have a balance formula among $\mu$, the leading term $a_0$ of the Alexander polynomial, and Bowen's $p$-adic entropy $h_p$ of the meridian action on the Alexander module \cite{Ueki4}, where we obviously have $\mu=0$ for knots by $\Delta_K(1)=\pm1$. 

In regard with $\mu$-invariants, we prove the relation of the Iwasawa $\mu$-invariants of $\Zp$-covers and $\Z/m\Z\times \Zp$-covers (Theorem \ref{thm.mu}) and that if the residual representation $\ol{\rho}=\rho \mod p$ satisfies $\Delta_{\ol{\rho},i}(t)\neq 0$, then $\mu=0$ for $\rho$ holds (Theorem \ref{thm.mu=0}). 
Another result of Friedl--Vidussi \cite[Theorem 1.1]{FriedlVidussi2013JEMS} yields that the set of $\mu$'s also determine the fiberedness of a knot. 
More precisely, we prove that $K$ is fibered if and only if $\mu$ is defined and $\mu=0$ holds for every ${\rm GL}_N$-representations (Theorem \ref{thm.fib.genus} (3)). 
For a non-fibered knot $K$, finding $\rho$ with $\mu>0$ is a difficult problem in general, 
though the latest breaking result of Morifuji--Suzuki \cite{MorifujiSuzuki2022IJM} 
finds a representation of $K=9_{35}$ with $\mu=4>0$ for $p=2$. 
In this view, we further establish the $\mu=0$ theorem for ${\rm SL}_2$-representations of twist knot $J(2,2n)$ which is non-fibered if $n\neq 0,\pm1$ by direct calculations (Theorem \ref{thm.twistknot}) and attach some observation on residually reducible irreducible representations (Theorem \ref{thm.resred}). 

Intuitive examples attached are $4_1=J(2,-2)$ and $5_2=J(2,4)$ (Examples \ref{eg.holonomy}, \ref{eg.fig8}, \ref{eg.5_2}).
We remark that most of our argument is applicable to any finitely presented group $\pi$ with a surjective homomorphism to $\Z$. 
At the end of this paper, we paraphrase several famous conjectures into our language and explain how our study could shed new light on arithmetic topology.  

\section{Twisted Alexander polynomials} 
Here, we define the twisted Alexander polynomials of knot group representations and recollect their basic properties. 
Let $K$ be a knot in $S^3$ with the knot group $\pi=\pi_1(S^3-K)$ 
and let $\alpha:\pi\surj t^\Z$ denote the abelianization map sending meridians to a formal generator $t$, 
so that ${\rm Ker}\alpha$ corresponds to the $\Z$-cover $X_\infty\to X=S^3-K$. 
Let $R$ be an integrally closed Noetherian domain and put $\Lambda_R=R[t^\Z]$, so that $\Lambda_R$ is also an integrally closed Noetherian domain. 
For a representation $\rho:\pi \to {\rm GL}_N R$, the $i$-th twisted Alexander module $\mca{A}_{\rho,i}$ stands for one of the following modules, for which Shapiro's lemma and Hopf's theorem yield natural isomorphisms; 
\[H_i(X,\rho\otimes \alpha)\cong H_i(X_\infty,\rho) \cong H_i({\rm Ker}\,\alpha, \rho)\cong H_i(\pi, \rho\otimes \alpha).\] 
By the conjugate action of $t^\Z$-action, $\mca{A}_{\rho,i}$ becomes a finitely generated $\Lambda_R$-module. 

\emph{The divisorial hull} $\wt{\mf{a}}$ of an ideal $\mf{a}$ of an integral domain $\Lambda$ is defined as the intersection of all principal \emph{fractional} ideals containing $\mf{a}$. 
Let $\Lambda_0={\rm fr}\Lambda$ denote the field of fractions of $\Lambda$. 
If we put $[\Lambda:\mf{a}]=\{x\in \Lambda_0\mid x\mf{a}\subset \Lambda\}$, then $\wt{\mf{a}}=[\Lambda:[\Lambda:\mf{a}]]$ holds \cite[$\S$1, Proposotion 1]{MR0260715}. 
\begin{lem}[{\rm \cite[Lemma 0.1]{Ueki10}}]
Let $\Lambda$ be an integrally closed Noetherian domain and suppose that 
$\P$ runs through the set of all height one prime ideals of $\Lambda$, 
so that each localization $\Lambda_\P$ is a DVR and $R=\cap_\P \Lambda_\P$ holds in $\Lambda_0$. 
\begin{itemize}
\item If $\mf{a}$ is an ideal of $\Lambda$, then $\wt{\mf{a}}=\cap_\P \wt{\mf{a}}_\P$ holds. 
\item If $\mf{a}$ is an ideal of $\Lambda$ and $S$ a multiplicative set, then $\wt{\mf{a}_S}=\wt{\mf{a}}_S$ holds.
\item If ideals $\mf{a}, \mf{b}, \mf{c}$ of $\Lambda$ satisfy $\mf{a}_\P=\mf{b}_\P\mf{c}_\P$ for every $\P$, then $\wt{\mf{a}}=\wt{\mf{b}}\wt{\mf{c}}$ holds. 
\end{itemize}
\end{lem} 

If the divisorial hull of the \emph{initial} Fitting ideal of $\mca{A}_{\rho,i}$ is a principal ideal, then its generator is called \emph{the $i$-th Alexander polynomial} and denoted by $\Delta_{\rho,i}(t) \in \Lambda_R$, 
which coincides with an order element ${\rm ord}_{\Lambda_R}\mca{A}_{\rho,i}$ if $R$ is a UFD. 
The equalities of elements in $\Lambda_R$ up to multiplication by units is denoted by $\dot{=}$\ , so that $\Delta_{\rho,i}(t)$ is defined up to $\dot{=}$\ . We have $\Delta_{\rho,i}(t)\neq 0$ if and only if $\mca{A}_{\rho,i}$ is a torsion $\Lambda_R$-module. 

Since the localization $R_\p$ is a DVR (hence a PID) for each 
height 1 prime ideal $\p$ of $R$, previously known results yield the following (cf. \cite{KirkLivingston1999T1, FriedlVidussi2011survey}). 
\begin{prop} \label{lem.Delta} Suppose that $\Delta_{\rho,i}(t)$ is defined for $i=0,1,2$. Then, 
\begin{itemize}
\item We have $\Delta_{\rho,0}(t)\neq 0$ and $\Delta_{\rho,i}(t)\,\dot{=}\,1$ for $i\geq 3$ {\rm \cite[Proposition 2 (1)]{FriedlVidussi2011survey}}. 
\item $\Delta_{\rho,0}(t)$ is monic {\rm \cite[Lemma 4 (1)]{FriedlVidussi2011survey}}.
\item If $\Delta_{\rho,1}(t)\neq 0$, then $\Delta_{\rho,2}(t)\,\dot{=}\,1$ {\rm \cite[Proposition 2 (5)]{FriedlVidussi2011survey}}. 
\item If $\det(I-\rho(g))\in R^\times$ for some $g\in {\rm Ker}\alpha$, then $\Delta_{\rho,0}\,\dot{=}\,1$ holds {\rm(cf. \cite[Lemma 4.11]{Turaev2001book})}. 
\end{itemize} 
\end{prop}

The Reidemeister torsion $\tau_{\rho\otimes \alpha} \in {\rm fr}\Lambda_R/\pm{\rm Im} \det\rho$ is defined with less indeterminacy than $\Delta_{\rho,i}(t)$'s and is calculated as the ratio of the determinants of two square matrices. 
The equality $\tau_{\rho\otimes \alpha} \,\dot{=}\, \Delta_{\rho,1}(t)/\Delta_{\rho,0}(t)$ holds up to multiplication by units in $\Lambda_R$. 

The polynomial $\Delta_{\rho,i}(t) \in R[t^\Z]$ is said to be \emph{monic} if the leading coefficient is a unit of $R$. 
The Reidemeister torsion $\tau_{\rho\otimes \alpha}$ is said to be \emph{monic} if there are polynomials $f(t),g(t)\in R[t^\Z]$ with $\tau_{\rho\otimes \alpha}=f(t)/g(t)$ such that 
the leading coefficients of both $f(t)$ and $g(t)$ belong to the set $\pm{\rm Im} \det\rho$ $\subset R^\times$. 
Now suppose that $\Delta_{\rho,1}(t)\neq 0$. 
By Proposition \ref{lem.Delta}, if $\tau_{\rho\otimes \alpha}$ is monic, then $\Delta_{\rho,1}(t)$ is monic. 
If $\pm{\rm Im} \det\rho = R^\times$ holds (e.g., $R=\Z$), then the converse holds. 

Now suppose that $R$ is the Dedekind domain $O=O_{F,S}$ and put $\ul{S}=\{q \in \Z \mid (q)= \mf{q}\cap \Z {\rm \ for\ some\ }\mf{q}\in S\}$. 
The norm map on the polynomial ring is defined by 
\[{\rm Nr}_{F/\Q}:O[t^\Z]\to \Z_{\ul{S}}[t^\Z];\ \ds f(t)\mapsto \prod_{\sigma:F\inj \C}f(t)^\sigma.\] 
Regard $\mca{A}_{\rho,i}$ as a $\Z_{\ul{S}}[t^\Z]$-module and put 
$\ul{\Delta}_{\rho,i}(t)={\rm ord}_{\Lambda_{\Z_{\ul{S}}}}\mca{A}_{\rho,i}$. 
If $\Delta_{\rho,i}(t)$ is defined, then $\ul{\Delta}_{\rho,i}(t)={\rm Nr}_{F/\Q}\Delta_{\rho,i}(t)$ holds. 
For a parabolic representation $\rho$ of a two-bridge knot group, $\ul{\Delta}_{\rho,i}(t)$ coincides with $\Delta_{\rho^{\rm tot}}(t)$ of the total representation $\rho^{\rm tot}:\pi\to {\rm GL}_N\Z$ (cf.\cite{SilverWilliams2009TA}). 
The polynomial $\ul{\Delta}_{\rho,i}(t)$ is defined even if $\Delta_{\rho,i}(t)$ is not defined. 
We may assume that $\ul{\Delta}_{\rho,i}(t)$ belongs to $\Z[t^\Z]$. 
We have ${\rm Nr}_{F/\Q}\tau_{\rho\otimes \alpha}(t)\,\dot{=}\,\ul{\Delta}_{\rho,1}(t)/\ul{\Delta}_{\rho,0}(t)$.

\section{Twisted cyclic resultants} 
For a polynomial $f(t)\in O[t^\Z]$ and $n\in \Z_{>0}$, the cyclic resultant is defined by 
\[{\rm Res}(t^n-1,f(t))=\prod_{\zeta^n=1}f(\zeta),\] 
which coincides with the determinant of the Sylvester matrix in $M_{n+{\rm deg}f}(O)$ so that ${\rm Res}(t^n-1,f(t))\in O$ holds. 
The order of a finite group $G$ is denoted by $|G|$. The torsion subgroup of an abelian group $G$ is denoted by $G_{\rm tor}$. The $p$-adic absolute value of an integer $mp^r$ with $m,r\in \Z$, $p\nmid m$ is defined by $|mp^r|_p=p^{-r}$. 
Let $\C_p$ denote the $p$-adic completion of an algebraic closure of the $p$-adic number field $\Q_p$ and fix an embedding $\ol{\Q}\inj \C_p$ of an algebraic closure $\ol{\Q}$ of $\Q$. Then the Mahler measure and the Gauss norm of $f(t)\in \Z[t^\Z]$ are defined by 
\[\ds \log \mh(f(t))=\int_{|z|=1}\log |f(z)|\dfrac{dz}{z}, \ \ \log \mh_p(f(t))=\lim_{n\to \infty} \sum_{\zeta^n=1}\dfrac{\log |f(\zeta)|_p}{n},\]
which are naturally interpreted even if $f(t)$ has a root on the unit circle. 
If $f(t)=\sum_i a_i t^i=a_0\prod_i(t-\alpha_i)$, then the Jensen's formula asserts that 
\begin{gather*}
\mh(f(t))=a_0\prod_i {\rm max}\{1,|\alpha_i|\},\\
\mh_p(f(t))=|a_0|_p\prod_i {\rm max}\{1,|\alpha_i|_p\}={\rm max}\{|a_i|_p\}_i
\end{gather*}
hold \cite{Ueki4}. 
In addition, we have the following. 

\begin{thm}[(cf.{\cite[Theorem 11.1]{Ueki10}})] \label{thm.Fox} 
Let $\rho:\pi\to {\rm GL}_NO$ be a knot group representation over $O=O_{F,S}$ 
with $\ul{\Delta}_{\rho,1}(t)\neq 0$. 
Then, notation being as above, the following holds. 

{\rm (1)} For each $n\in \Z_{>0}$, the Wang exact sequence induces a natural isomorphism 
\[\mca{A}_{\rho,0}/(t^n-1)\mca{A}_{\rho,0} \congto H_0(X_n,\rho)\] 
and a natural short exact sequence 
\[0\to \mca{A}_{\rho,1}/(t^n-1)\mca{A}_{\rho,1} \overset{p_{1,n}}{\To} H_1(X_n,\rho) \overset{\partial}{\To} {\rm Ker}(t^n-1|_{H_0(X_\infty,\rho)}) \to 0.\] 
The group ${\rm Ker}(t^n-1|_{H_0(X_\infty,\rho)})$ is a finite group whose size is bounded if $n$ is regarded as a variable, and is trivial if ${\rm Fitt}\mca{A}_{\rho,0}=\wt{{\rm Fitt}}\mca{A}_{\rho,0}$ holds. 

{\rm (2)} Put $\Psi_n(t)={\rm gcd}(t^n-1, \ul{\Delta}_{\rho,1}(t))$ and $r_n={\rm Res}((t^n-1)/\Psi_n(t),\ul{\Delta}_{\rho,1})$. Then 
\[|H_1(X_n,\rho)_{\rm tor}|=c_n k_n|r_n|\prod_{p\in \ul{S}}|r_n|_p\]
holds, where $c_n$ is a bounded sequence defined by $c_n=|(\mca{A}_{\rho,1}/\Psi_n(t)\mca{A}_{\rho,1})_{\rm tor}|$ and $k_n$ is another bounded sequence. 
In addition, for each prime number $p$, the $p$-torsion satisfies \[||H_1(X_n,\rho)_{\rm tor}||_p=|c_n k_n r_n|_p.\]
If ${\rm Fitt}\mca{A}_{\rho,i}=\wt{{\rm Fitt}}\mca{A}_{\rho,i}$ for $i=0,1$, then $(k_n)_n$ is trivial.  

The similar equalities hold for $\Delta_{\rho,0}(t)$ and $|H_0(X_n,\rho)_{\rm tor}|$. If ${\rm Fitt}\mca{A}_{\rho,0}=\wt{{\rm Fitt}}\mca{A}_{\rho,0}$, then $(k_n)_n$ is trivial. 

{\rm (3)} The following asymptotic formulas of Mahler measures hold:  
\[\ds \lim_{n\to \infty}|H_i(X_n,\rho)_{\rm tor}|^{\frac{1}{n}}=\mh(\ul{\Delta}_{\rho,i}(t)), \ 
\ds \lim_{n\to \infty}||H_i(X_n,\rho)_{\rm tor}||_p^{\frac{1}{n}}=\mh_p(\ul{\Delta}_{\rho,i}(t)).\]
\end{thm} 

\begin{proof} The structure theorem of finitely generated torsion modules over an integrally closed Noetherian domain \cite[Theorem 2.36]{Ochiai2014-Iwasawa1} yields a homomorphism 
\[\varphi_i: H_i(X_\infty,\rho)\overset{\sim}{\to} \mca{M}_i=\bigoplus_j O[t^\Z]/\p_{i,j}^{r_{i,j}}\]
to a standard module $\mca{M}_i$ with finite kernel and cokernel, where $(\p_{i,j})_j$ is a finite sequence of height 1 prime ideals of $O[t^\Z]$ and $r_{i,j}\in \Z_{>0}$, 
hence a finite sequence $(\mf{m}_j)_j$ of maximal ideals of $O[t^\Z]$ satisfying  ${\rm Fitt}H_i(X_\infty,\rho)=\wt{\rm Fitt}H_i(X_\infty,\rho)\prod_j \mf{m}_{i,j}$. If $\{\mf{m}_{i,j}\}_j$ is empty, then $\varphi_i$ is an isomorphism. 

The Wang exact sequence yields the following long exact sequence; 
\[\cdots \overset{\partial}{\To} H_1(X_\infty,\rho) \overset{t^n-1}{\To}H_1(X_\infty,\rho) \overset{p_{1,n}}{\To}H_1(X_n,\rho)\ \ \ \ \]
\[\ \ \ \ \overset{\partial}{\To} H_0(X_\infty,\rho) \overset{t^n-1}{\To}H_0(X_\infty,\rho) \overset{p_{0,n}}{\To}H_0(X_n,\rho) \to 0.\]
Now standard arguments of cyclic resultants \cite[Theorem 3.13]{Hillman2} and Mahler measures yield the assertions (see \cite[Theorem 11.1]{Ueki10}). 
\end{proof}

Now let $\p$ be a maximal ideal of $O$ over $p$ and let $O_\p$ denote the ring $\varprojlim O/\p^rO$ of $\p$-adic integers. 
We may assume that $\p \not \in S$, since otherwise the $p$-torsion of an $O$-module is always trivial. 
Define the norm map on $O_\p[t^\Z]$ by 
\[{\rm Nr}_{F_\p/\Qp}:O_\p[t^\Z]\to \Zp[t^\Z];\ \ds f(t)\mapsto \prod_{\sigma:F_\p\inj \C_p}f(t)^\sigma,\] 
and put $\ul{\Delta}_{\rho_\p,i}(t)={\rm Nr}_{F_\p/\Qp} \Delta_{\rho_\p,i}(t) \in \Zp[t^\Z]$. 
Then we obtain the following as well. 

\begin{thm} 
If any $\p$-adic representation $\rho_\p:\pi\to {\rm GL}_N O_\p$ of a knot group is given, then 
the similar assertions to Theorem \ref{thm.Fox} (1), (2), and the equation on $\mh_p$ in (3) hold. 
\end{thm}

Suppose again that $\rho:\pi\to {\rm GL}_N O$ is a knot group representation with $\ul{\Delta}_{\rho,1}(t)\neq 0$. Then for each $p$, the coefficient decomposition $H_i(X_n,\rho\otimes \Zp)\cong \bigoplus_{\p|p} H_1(X_n,\rho\otimes O_\p)$ yields the equality  
\[\ul{\Delta}_{\rho,i}(t)=\prod_{\p|p}{\rm Nr}_{F_\p/\Qp} \Delta_{\rho\otimes O_\p,i}(t)= \prod_{\p|p} \ul{\Delta}_{\rho\otimes O_\p,i}(t)\] in $\Zp[t^\Z]$. Thus the two theorems above are related via $\prod_{\p|p}$. 

\section{Twisted Iwasawa invariants} 
Let $O=O_{F,S}$ and $O_\p$ be as before and let $0\neq m\in \Z$ coprime to $p$. 
We assume $p\not \in \ul{S}$, since if otherwise the $p$-torsion of an $O$-module is always trivial. 
Let $\rho:\pi\to \GL_NO$ be a knot group representation with $\Delta_{\rho,1}(t)\neq 0$. 
Note that for a fixed $m$ we have $(t^{mp^r}-1) \mid (t^{mp^{r+1}}-1)$, so the sequences  $(c_{mp^r})_r$ and $(k_{mp^r})_r$ in Theorem $\ref{thm.Fox}$ (2) are both constants for $r\gg0$. 
In addition, by the natural isomorphism in Theorem $\ref{thm.Fox}$ (1) and 
a standard argument verifying the Mittag-Leffler condition (cf. \cite[Theorem 4.9]{Ueki2}), 
we see that the Iwasawa module defined by \[\mca{A}_{\rho,i,\p}=\varprojlim_r H_i(X_{mp^r},\rho\otimes O_\p)\] is 
a finitely generated torsion $O_\p[[t^{(\Z/m\Z)\times \Zp}]]$-module. 
Now the Iwasawa isomorphism \[O_\p[[t^{\Zp}]]\cong O_\p[[T]]; t\mapsto 1+T,\]  
the character decomposition of the Iwasawa module over 
$O_\p[[t^{(\Z/m\Z)\times \Zp}]]=O_\p[t^{\Z/m\Z}][[t^{\Zp}]]$ \cite[pp.291--292]{Washington}, 
the $p$-adic Weierstrass preparation theorem \cite[Theorem 7.3]{Washington}, 
and a standard argument (cf. \cite[Sections 13.2, 13.3]{Washington} for $O_\p=\Zp$ and $\Psi_n=1$, 
\cite[Theorem 2.45, Remark 2.50]{Ochiai2014-Iwasawa1} for general cases)
together with the decomposition $\ul{\Delta}_{\rho,i}(t)=\prod_{\p|p}\ul{\Delta}_{\rho\otimes O_\p,i}(t)$ 
yield the following assertion.

\begin{thm}[(The Iwasawa type formula)] \label{thm.Iwasawa} 
Let $\rho:\pi\to {\rm GL}_NO$ be a knot group representation with $\Delta_{\rho,1}(t)\neq 0$ 
and let $\lambda, \mu \geq 0$ denote the unique integers satisfying 
\[\prod_{\zeta^m=1} \ul{\Delta}_{\rho, i}(\zeta(1+T))\, \dot{=}\, p^\mu(T^\lambda+p(\text{lower\ terms}))\ {\text in}\ \Zp[[T]].\]
Then there exits an integer $\nu$ such that for any integers $r\gg 0$ the following equality holds:
 \[ ||H_i(X_{mp^r},\rho)||_p=p^{-(\lambda r+\mu p^r +\nu)}.\] 

The assertion persists if $\rho$ is replaced by $\rho_\p:\pi\to {\rm GL}_NO_\p$. 
If $\lambda_\p, \mu_\p, \nu_\p$ denote the invariants of $\rho\otimes O_\p$ for each $\p|p$, then $\lambda=\sum_{\p|p} \lambda_\p$, $\mu=\sum_{\p|p} \mu_\p$, and $\nu=\sum_{\p|p} \nu_\p$ hold. 
\end{thm} 

We call these $\lambda,\mu,\nu$ the $i$-th Iwasawa invariants of the representation $\rho$ and the $\Z/m\Z\times \Zp$-cover. We write $\lambda=\lambda_p=\lambda_i=\lambda_{p,i}$ if we needed to clarify. 
Now recall that the Reidemeister torsion satisfies 
$\tau_{\rho\otimes \alpha}(t)=\Delta_{\rho,1}(t)/\Delta_{\rho,0}(t)$. 
Again by the $p$-adic Weierstrass preparation theorem, we have 
\[\prod_{\zeta^n=1}{\rm Nr}_{F/\Q}\tau_{\rho\otimes \alpha}(\zeta(1+T))\,\dot{=}\,p^{\mu_\tau} (T^{\lambda_\tau}+p(\text{lower\ terms}))\] in $\Zp[[T]]$ 
for the integers $\mu_\tau=\mu_1-\mu_0$, $\lambda_\tau=\lambda_1-\lambda_0$. We further put $\nu_\tau=\nu_1-\nu_0$ and call them the Reidemeister--Iwasawa invariants. They essentially convey topological information. 
We remark that $\mu$ and $\lambda$ of $\rho$ depend on $p$, since $\mu$ is the $p$-exponent of the polynomials in both $t$ and $T$, and $\lambda$ counts the roots of the polynomial in $t$ on the $p$-adic unit disc $\{z\in \C_p\mid |z-1|_p\leq 1\}$. 

If $\rho:\pi\to {\rm GL}_NO$ is given and the residual representation $\ol{\rho}=\rho\mod \p:\pi\to {\rm GL}_N(O/\p)$ is irreducible and non-abelian, then 
we have $\Delta_{\rho, 0}(t)$ $\dot{=}$ $1$ and $\Delta_{\rho,1}(t)$ $\dot{=}$ $ \tau_{\rho\otimes\alpha}(t)$ in $\Lambda_{O_\p}=O_\p[t^\Z]$ by \cite[Corollary 3]{TangeRyoto2018JKTR}, so that 
$\mu_0=\lambda_0=0$, $\mu_\tau=\mu_1$, and $\lambda_\tau=\lambda_1$ hold.

\begin{eg} \label{eg.holonomy} 
Let $K=4_1$ (the figure-eight knot) in $S^3$ and consider the two lifts $\rho_{\pm}:\pi\to \GL_2 \C$ of the holonomy representation, 
which may be regarded as representations over the ring $O=\Z[\frac{1+\sqrt{-3}}{2}]$ of integers of a quadratic field. 
Then we have $\tau_{\rho_{\pm}\otimes \alpha}(t)=\Delta_{\rho_\pm,1}(t)=t^2\pm4t+1$. 
For $\rho_+$, we have $\Delta_{\rho_+,1}(1+T)=T^2+6T+6$, and hence $\lambda_2/2=\lambda_3/2=2$, $\lambda_p=0$ $(p\neq 2,3)$. 
For $\rho_-$, we have $\Delta_{\rho_-,1}(1+T)=T^2-2T+2$, and hence $\lambda_2/2=2$, $\lambda_p=0$ $(p\neq 2)$. 
In both cases, we have $\mu=0$ for any $p$. 
Since each lift corresponds to a spin structure, we may say that the Iwasawa $\lambda$-invariants of the $\Zp$-covers distinguish the spin structures of $K=4_1$. 
\end{eg}

\section{The degree and the $\lambda$-invarians}
Here we investigate the relation between the degree of the Alexander polynomial and the Iwasawa $\lambda$-invariants.
A Laurent polynomial $f(t)\in R[t^\Z]$ is said to be of degree $d$ if $f(t)=\sum_{j\leq i\leq j+d} a_i t^i$ with $a_j,a_{j+d}\neq 0$.

For a fixed $m\in \Z$ with $p\nmid m$, \emph{the $\Z/m\Z\times \Zp$-cover} of a knot is the compatible system ($X_{mp^r}\to X)_r$ of the $\Z/mp^r\Z\cong \Z/m\Z\times \Z/p^r\Z$-covers obtained from the $\Z$-cover via the natural surjective homomorphism $\Z\surj \Z/mp^r$.  
In what follows, we always assume that the $i$-th Alexander polynomial of a given representation is defined. 
The relation between $\lambda$'s and the degrees of the polynomials is given by the following assertion: 

\begin{thm} \label{thm.degree} 
Let $\rho:\pi\to {\rm GL}_NO$ be a representation with $\Delta_{\rho,1}(t)\neq 0$. Then for any $m$ with $p\nmid m$, the $\lambda$-invariant of $\rho$ and $\Z/m\Z\times \Zp$-cover satisfies the inequality 
\[\lambda_i \leq \deg \ul{\Delta}_{\rho,i}(t).\] 
For all but finitely many $p$, there exists some $0\neq m' \in \Z$ with $p\nmid m'$ such that for any $m\in \Z$ with $p\nmid m$ and $m'|m$, the equality $\lambda={\rm deg}(\ul{\Delta}_{\rho,i}(t))$ holds. 

If instead $\rho_\p:\pi\to {\rm GL}_NO_\p$ is given, then $\lambda_i \leq \deg \ul{\Delta}_{\rho_\p,i}(t)$ holds. 
\end{thm} 

\begin{proof} Put $\Delta(t)=\ul{\Delta}_{\rho,i}(t)$. Take an algebraic closure $\ol{\Q}$ of $\Q$. For each prime number $p$, let $\C_p$ denote the completion of an algebraic closure of $\Qp$ and fix an embedding $\ol{\Q}\inj \C_p$. 
Recall that by the $p$-adic Weierstrass preparation theorem and the strong triangle inequality, the image of $\Delta(t)$ via \[\Zp[t^\Z]\to \Zp[[T]]; t\mapsto 1+T\] only knows the roots of $\Delta(t)$ on the $p$-adic unit disc $|z-1|_p\leq 1$. 
Let $g$ denote the gcd of the coefficients of $\Delta(t)$. 
If the leading coefficient of $\Delta(t)/g$ is not divisible by $p$, then all roots lie on the $p$-adic unit circle $|z|_p=1$ in $\C_p$. For each $\alpha \in \ol{\Q}$ with $|\alpha|_p=1$, there exists a unique $p$-prime root $\zeta$ of unity in $\ol{\Q}$ such that $|\alpha-\zeta|_p<1$ holds. 
Write $\Delta(t)=\prod_j (t-\alpha_j)$ and suppose that each $\alpha_j$ corresponds to a primitive $m'_j$-th roots of unity. 
Since the polynomial of $\Z/m\Z\times \Zp$-cover is given by $\prod_{\zeta^m=1}\Delta(\zeta t)$, by putting $m'={\rm gcd}\{m'_j\}_j$, we obtain $\lambda=\deg \Delta(t)$ for any $m$ with $m'|m$. 
If the leading coefficient of $\Delta(t)/g$ is divisible by $p$, then $\Delta(t)$ has divisors $t-\alpha$ with $|\alpha|_p>1$, 
hence so does every $\Delta(t\zeta)$, while other roots behave as in the previous case. Thus we obtain the first assertion. 
The second assertion is proved by the same argument. 
\end{proof} 

Note that in the statement of Theorem \ref{thm.degree} the term $\deg \ul{\Delta}_{\rho,i}(t)$ appears, but not the term $\deg \prod_{\zeta^m=1}\ul{\Delta}_{\rho,i}(\zeta t)$.  
We refer to the previous article \cite{Ueki4} for all basic facts of $p$-adic numbers used in the proof above. 
A similar method was used in \cite{Ueki6} to investigate the profinite rigidity of twisted Alexander polynomials.

If we put $N=1$ and $\rho=1$, then $\Delta_{\rho,1}(t)=\Delta_K(t)$ is the classical Alexander polynomial. 
It is classically known that if $K$ is a fibered knot, then $\Delta_K(t)$ is monic and its degree coincides with twice the degree, yielding the following. 

\begin{thm} \label{thm.fiberedknot} 
If $K$ is a fibered knot, then the Iwasawa $\lambda$-invariants of $\Z/m\Z\times \Zp$-covers with $i=1$ and $\rho=1$ determines the genus $g(K)$, that is, we have $\lambda_1\leq 2g(K)$ and there is some $m'$ such that $m'|m$ implies the equality. 
\end{thm}

\begin{eg} \label{eg.fig8} 
Consider the $\Z/m\Z\times \Zp$-covers of the figure-eight knot $K=4_1$, which is a hyperbolic fibered knot with genus 1. 

If $m=1$, then $\Delta_K(t)=t^2-3t+1$, $\Delta_K(1+T)=T^2-T-1$ $\dot{=}$ $1$ in $\Zp[[T]]$ and $\lambda_p=0$ for any $p$.

If $m=2$, then 
\begin{align*}
\prod_{\zeta^2=1}\Delta_K(\zeta t) 
=(t^2-3t+1)(t^2+3t+1)
\overset{t\mapsto T+1}{=}(T^2-T-1)(T^2+5T+5)
\,\,\dot{=}\,\,T^2+5T+5
\end{align*}  
in $\Zp[[T]]$, $\lambda_5=2$, $\lambda_p=0$ ($p\neq 5$). 

If $m=4$, then 
\begin{align*}
\prod_{\zeta^4=1}\Delta_K(\zeta t)&=(t^2-3t+1)(t^2+3t+1)(t^4+7t^2+1) \\ 
& \hspace{-10mm} \overset{t\mapsto T+1}{=} (T^2-T-1)(T^2+5T+5)(T^4 + 4 T^3 + 13 T^2 + 18 T + 9)\,\,\dot{=}\,\,T^2+5T+5
\end{align*} 
in $\Zp[[T]]$, 
$\lambda_5=2$, $\lambda_p=0$ ($p\neq 5$). 
\end{eg}

\section{Monicness} 
On the Iwasawa $\mu$-invariants and the monicness of the polynomials, we have the following. 

\begin{thm} \label{thm.mu} Let $\rho$ or $\rho_\p$ be as in Theorem \ref{thm.degree} and put $\ul{\Delta}(t)=\ul{\Delta}_{\rho,i}(t)$ or $\ul{\Delta}_{\rho_\p,i}(t)$. Then for a $\Z/m\Z\times \Zp$-cover, 
\[p^{-\mu}=\mh_p(\prod_{\zeta^m=1} \ul{\Delta}(\zeta t))=\mh_p(\ul{\Delta}(t))^m\] holds. 
Let $\mu'$ denote the $\mu$-invariant of the $\Zp$-cover. Then $\mu=m\mu'$ holds, 
so that $\mu=0$ is equivalent to $\mu'=0$.  
\end{thm} 

\begin{proof} Noting $\mh_p(\ul{\Delta}(t))=\mh_p(\ul{\Delta}(\zeta t))$ for each $\zeta$ with $\zeta^m=1$, 
the comparison between the latter half of Theorem \ref{thm.Fox} (3) and Theorem \ref{thm.Iwasawa} yields the assertion. 
\end{proof} 

\begin{thm} \label{thm.monic} 
If $p$ and $m$ are regarded as indeterminacies, then the set of $\mu$'s and $\lambda$'s determine whether $\Delta_{\rho,i}(t) \in O[t^\Z]$ is monic in the following sense; 
$\mu_p=0$ holds for all $p$ and the sequence $({\rm max}\,\{\lambda_{p,m}\mid m\}\,)_p$ is a constant if and only if $\Delta_{\rho,i}(t)$ is monic. 
\end{thm} 

\begin{proof} We see that $\Delta_{\rho,i}(t)$ is monic if and only if $\ul{\Delta}_{\rho,i}(t)$ is monic. 
The gcd of the coefficients of $\ul{\Delta}_{\rho,i}(t)$ is given by $g=\prod_p p^{\mu_p}$, where $\mu_p$'s denote the $\mu$-invariants of the $\Zp$-covers. 
The leading term of $\ul{\Delta}_{\rho,i}(t)/g$ is divisible by $p$ if and only if for any $m$ the $\lambda$-invariant of $\Z/m\Z\times \Zp$ is less than $\deg \ul{\Delta}_{\rho,i}(t)/g$. The degree $\deg \ul{\Delta}_{\rho,i}(t)/g$ is determined by $\lambda$'s by Theorem \ref{thm.degree}. This completes the proof. 
\end{proof}

\begin{eg} \label{eg.5_2}
The knot $K=J(2,4)=5_2$ is a non-fibered hyperbolic twist knot with a non-monic classical Alexander polynomial $\Delta_K(t)=2t^2-3t+2$. 

If $p=2$, we have $\Delta_K(t)=t\,\dot{=}\,1$ in $\F_2[t^\Z]$. Hence for any $m>0$ we have $\prod_{\zeta^m=1}\Delta_K(\zeta t)\,\dot{=}\,1$ and $\lambda_2=0$. 

If $p\neq 2$, then $\Delta_K(1+T)$ is monic. Hence there exists some $0\neq m' \in \Z$ such that for any $m\in \Z$ with $m'|m$, the degree of $\deg \prod_{\zeta^m=1} \Delta_K(\zeta(1+T))=2$ is 2 
and $\lambda_p=2$ holds. 
If $p=3$ for instance, then we have $\Delta_K(t)=2(t^2+1)$ in $\F_3[t^\Z]$. Hence for any $m$ with $4|m$ the degree of $\prod_{\zeta^m=1} \Delta_K(\zeta(1+T))$ is 2 and $\lambda_3=2$ holds. 

Thus, $\lambda$'s determine that $\Delta_K(t)$ is non-monic, hence that $K=5_2$ is not fibered. 
\end{eg}

\section{Fiberedness, genus, and profinite rigidity} 
Let $x_K$ denote the Thurston norm of the abelianization map $\alpha:\pi\surj \Z$ of the knot group $\pi$ of a knot $K$, so that we have $x_K=1$ if $K$ is the unknot and $x_K=2g(K)-1$ if otherwise. 
Here we recollect implications for knot group representations of deep results due to Friedl, Kim, Vidussi, and Nagel, that are based on works of Agol, Przytycki, Wise, and others.  
(For (1) and (3), see also \cite[Theorem 3.3]{BoileauFriedl2020AMS}.)
\begin{prop} \label{prop.FV} 

{\rm (1) \cite[Theorem 1.2]{FriedlVidussi2015Crelle}} 
{\rm (i)} 
Let $\ol{\rho}:\pi\surj G$ be a surjective representation onto a finite group with $0\neq \Delta_{\ol{\rho},1}(t)$ $\in \Z[t^\Z]$. Then 
\[x_K\geq {\rm max}\{0, \frac{1}{|G|}(-{\rm deg}(\Delta_{\ol{\rho},0}(t)+{\rm deg}\Delta_{\ol{\rho},1}(t)-{\rm deg}\Delta_{\ol{\rho},2}(t))\}.\]

{\rm (ii)} There is a surjective representation $\ol{\rho}:\pi\surj G$ onto a finite group  with $0\neq \Delta_{\ol{\rho},1}(t)$ $\in \Z[t^\Z]$ satisfying 
\[x_K = {\rm max}\{0, \frac{1}{|G|}(-{\rm deg}(\Delta_{\ol{\rho},0}(t)+{\rm deg}\Delta_{\ol{\rho},1}(t)-{\rm deg}\Delta_{\ol{\rho},2}(t))\}.\] 

Similar assertions hold for $\ol{\rho}$'s over a finite field $\F_p$ 
for any fixed prime number $p$, 
in which cases we have $\Delta_{\ol{\rho},i}(t) \in \F_p[t^\Z]$. 

{\rm (2) \cite[Theorem 1.1]{FriedlKim2008}+\cite[Theorem 1.2]{FriedlVidussi2011AnnMath}} 
The knot $K$ is fibered if and only if for every surjective representation $\ol{\rho}:\pi\surj G$ onto any finite group, 
$\Delta_{\ol{\rho},1}(t)$ is monic in $\Z[t^\Z]$ and the equality in {\rm (1) (ii)} holds. 

{\rm (3) \cite[Theorem 1.1]{FriedlVidussi2013JEMS}+\cite[Theorem1.3]{FriedlNagel2015Illinois}} 
The knot $K$ is fibered if and only if every surjective representation 
$\ol{\rho}:\pi\to G$ onto a finite group 
satisfies $0\neq \Delta_{\ol{\rho},1}(t)$ in $\Z[t^\Z]$. 
A similar assertion holds for $\ol{\rho}$ over $\F_p$ for almost all $p$. 
\end{prop} 

Friedl--Kim \cite[Theorem 1.1]{FriedlKim2008} (as well as \cite{Cha2003} and \cite{GodaKitanoMorifuji2005}) asserts that if a knot is fibered then every $\Delta_{\ol{\rho},1}(t)$ is monic and the equation on the degree holds, so 
(3) implies (2) for $\Delta_{\ol{\rho},1}(t)$ in $\Z[t^\Z]$. The version of (3) over $\F_p$ is due to Friedl--Nagel and it played an important role in proving the profinite rigidity of the fiberedness of knots in \cite{BoileauFriedl2020AMS}.  

Note that a surjective representation $\ol{\rho}$ onto a finite group $G$ may be replaced by 
a representation to ${\rm Aut}\Z[G]\cong {\rm GL}_{|G|}\Z$, that induces a representation to ${\rm GL}_{|G|}\F_{p^r}$ for any prime number $p$ and any $r\in \Z_{>0}$. 
Any $\ol{\rho}:\pi\to {\rm GL}_N\F_p$ lifts to $\rho:\pi\to {\rm GL}_NO$ for $O=\Zp$ or its finite extension. 
Since $\Delta_{\ol{\rho},1}(t)=\Delta_{\rho,1}(t)$ mod $p$, 
we have $\Delta_{\ol{\rho},1}(t)= 0$ if and only if $p| \Delta_{\rho,1}(t)$. 
Thus, if there exists a lift $\rho$ with $\Delta_{\rho,1}(t)\neq 0$, then we obtain $\mu_1>0$. 
Hence by Proposition \ref{prop.FV} and Theorems \ref{thm.degree}, \ref{thm.monic}, we obtain the following. 

\begin{thm} \label{thm.fib.genus} 
The twisted Iwasawa invariants of representations $\rho:\pi\to {\rm GL}_N O_\p$ and $\Z/m\Z\times \Zp$-covers 
determine the fiberedness and the genus of any knot. 
More precisely, for each $\rho:\pi\to {\rm GL}_N O_\p$, where $\p$ is a prime ideal of $O$ over a prime number $p$, put $d={\rm deg}O_\p/\Zp$. 

{\rm (1)} 
For almost every prime number $p$, 
if $\rho$ and $m$ are regarded as indeterminacies, 
then the set of $\lambda_\tau=\lambda_1-\lambda_0$ 
determines the genus $g(K)$ of a knot in the sense that 
\[x_K={\rm max}\{0, \lambda_\tau/Nd \mid \rho, m\}\] holds. 
In addition, $x_K={\rm max}\{0, \lambda_\tau/Nd \mid \rho, m, p\}$ holds. 

{\rm (2)} Let $p$ be a fixed prime number. 
The knot $K$ is fibered if and only if for any representation $\rho:\pi\to {\rm GL}_NO_\p$, 
$\mu_1$ and $\lambda_1$ are defined, $\mu_1=0$ holds, and the set of $\lambda_\tau$'s determines $g(K)$ 
in the sense that 
\[x_K={\rm max}\{0, \lambda_\tau/Nd \mid m\}\] holds. 

{\rm (3)} 
Let $p$ be a fixed prime number. 
The knot $K$ is fibered if and only if for every $\rho:\pi\to {\rm GL}_NO_\p$, $\mu_1$ is defined and $\mu_1=0$ holds. 
\end{thm} 

It is known that if $K$ is a torus knot, which is a fibered knot, then every irreducible ${\rm SL}_2$-representation has the same $\Delta_{\rho}(t)$. However, it seems that for a generic non-fibered knot $K$, $\Delta_{\rho,1}(t)$ is not a constant on the character variety, and a generic lift $\rho$ of $\ol{\rho}$ with $\Delta_{\ol{\rho},1}(t)=0$ satisfies $\Delta_{\rho,1}(t)\neq 0$. 
We optimistically attach the following conjecture. 
\begin{conj} \label{conj.mu}
Let $p$ be a fixed prime number. If $K$ is non-fibered, then there exists some $\rho:\pi\to {\rm GL}_NO_\p$ such that $\mu_1$ is defined and $\mu_1>0$ holds. 
\end{conj} 


In general, it is a difficult problem to find a representation $\rho$ that detects the non-fiberedness of a knot. In this view, we study ${\rm SL}_2$-representations of twist knots in Sections 8, 9. 
Recently, Morifuji and Suzuki found several examples in a systematic manner \cite{MorifujiSuzuki2022IJM}. Among other things, they found that $K=9_{35}$ admits a ${\rm GL}_6(\F_2)$-representation $\ol{\rho}$ with $\Delta_{\ol{\rho}}(t)=0$, as well as this $\ol{\rho}$ lifts to a ${\rm GL}_6(\Z)$-representation $\rho$ such that $\Delta_{\rho}(t)=2^4(t+1)^2(t-1)^2(7t^2+13t+7)(7t^2-13t+7)$. 
This means that we have $\mu=4>0$ for $p=2$. 
We wonder if every lift of $\ol{\rho}$ to $\Z_2$ satisfies $\mu=4$, 
and if this concrete value $\mu=4$ conveys any information. 

The ${\rm GL}_2$-detection theory of the Thurston norms has been developed by Agol--Dunfield \cite{AgolDunfield2020AMS} 
from a viewpoint of the Dunfield--Friedl--Jackson conjecture, which claims that information of hyperbolic knots may be conveyed by the twisted Alexander polynomials of the holonomy representations \cite{DunfieldFriedlJackson2012}. 

The detection of the fiberedness and the genera is based on the fact that the estimate by the degrees of the Alexander polynomials is sharp. 
Since the cyclic resultants are profinite invariants (cf.~\cite[Remark 4.2]{Ueki5}), so are the Iwasawa invariants, noting that our Iwasawa invariants are defined if and only if the Iwasawa module $\mca{A}_{\rho,i}$ is a torsion $O_\p[[t^{\Zp}]]$-module.   
Hence we obtain the following result on the profinite rigidity, refining \cite[Theorem 1.2]{BoileauFriedl2020AMS} and \cite{Ueki7}. 

\begin{thm} \label{thm.profinite} 
The fiberedness and the genus of a knot $K$ are determined by the isomorphism class of the profinite completion of the knot group $\pi$, via the twisted Iwasawa invariants. 
\end{thm} 

\begin{proof} 
The profinite completion of the group $\pi$ is defined by $\wh{\pi}=\varprojlim_{\Gamma \lhd \pi} \pi/\Gamma$, where $\Gamma$ runs through subgroups of finite indices. 
Note that if an unknown isomorphism $\wh{\pi}\cong \Pi$ to a profinite group $\Pi$ is given, we cannot detect the image of $\pi$ in $\Pi$. Thus we need to take into account all representations over $O_\p$ and surjective homomorphisms $\alpha:\wh{\pi}\surj t^{\wh{\Z}}$, where $\wh{\Z}=\varprojlim \Z/n\Z$ denotes the Pr\"uffer ring. 
However, for an arbitrary taken generator $s \in t^\wh{\Z}$, we have $s^v=t$ for some unit $v\in \wh{\Z}^*$ and the Alexander polynomial is of the form $\Delta_{\rho,i}(s^v)$. Since the cyclic resultants persist under $t\mapsto t^v$, it suffices only to consider the Iwasawa invariants of the covers obtained from the $\Z$-cover. 
Now Theorem \ref{thm.fib.genus} yields the assertion. 
\end{proof}

We remark that the profinite rigidity of the fiberedness of a general compact 3-manifold is proved by Jaikin-Zapirain \cite{JaikinZapirain2020GT}. 
Wilkes proved among other things that if $J$ is a prime graph knot and $K$ is another knot, then $\wh{\pi}_J\cong \wh{\pi}_K$ implies that $S^3-J$ is homeomorphic to $S^3-K$ \cite[Theorem B]{Wilkes2019Israel}. 
Even though there have been made much progress by many researchers, it seems that we still have a lot to do in this direction (cf. \cite{BJZR2023.problems}).  

\section{$\mu=0$ for ${\rm SL}_2$-representations of twist knots} 
Theorem \ref{thm.fib.genus} (3) asserts that if $K$ is not fibered, then there is some $\rho$ with 
$\Delta_\rho(t)=0$ or with  
$\mu> 0$. However, finding such a representation is difficult and remains as an interesting problem. 
In this section, we establish the $\mu=0$ theorem for any ${\rm SL}_2$-representation of twist knots $J(2,2n)$ with $n\in \Z$, which is fibered only for $n=0,\pm 1$. 
We continue to suppose that $O=O_{F,S}$ and $\F=O/\p$ for a maximal ideal $\p$ of $O$ above a prime number $p$. 
We first prove a slightly general assertion: 

\begin{thm} \label{thm.mu=0} 
Let $\rho:\pi\to {\rm SL}_2O$ be a knot group representation. 
If the residual representation $\ol{\rho}=\rho \mod \p$ is irreducible and $\Delta_{\ol{\rho},1}(t)/\Delta_{\ol{\rho},0}(t)\neq 0$ in $\F[t^\Z]$, then $\mu_{p,1}=0$ holds. 
If instead $\mu_{p,1}\neq 0$, then $\ol{\rho}$ is non-acyclic. 
The similar assertion holds for $\rho_\p: \pi\to {\rm SL}_2O_\p$. 
\end{thm} 

\begin{proof} By $\Delta_{\ol{\rho},i}(t)=\Delta_{\rho,i}(t) \mod p$ and $p^{-\mu_p}=\mh_p(\ul{\Delta}_{\rho,i}(t))^m$, 
$\Delta_{\ol{\rho},i}(t)\neq 0$ in $\F[t^\Z]$ implies $\mu_{p,i}=0$. 
If $\ol{\rho}=\rho \mod \p$ is irreducible, then by $\Delta_{\rho,0}(t)$ $\dot{=}$ $1$, we have $\mu_{p,0}=0$ and $\mu_{p,\tau}=\mu_{p,1}$. Hence if $\Delta_{\ol{\rho},1}(t)/\Delta_{\ol{\rho},0}(t)\neq 0$, then $\mu_{p,\tau}=\mu_{p,1}=0$ holds. 

Recall that $\rho$ is said to be non-acyclic if $H_i(\pi,\rho)\neq0$ for some $i$. 
The conjugacy classes of non-acyclic irreducible representations correspond to the zeros of the acyclic torsion function $\tau(x,y)=\Delta_{\rho,1}(1)/\Delta_{\rho,0}(1)$ on some components of the character variety. 
If $\mu_{p,1}\neq 0$, then $\Delta_{\rho,1}(1)=0$ in $\F$, and hence $\ol{\rho}$ is non-acyclic. 
\end{proof}

Now we investigate the Iwasawa $\mu$-invariants of twist knot representations. 
For each $n\in \Z$, \emph{the twist knot} $K=J(2,2n)$ is defined by the following diagram, 
where the full twist is right-handed if $n>0$ and left-handed if $n<0$; 
\begin{center} 
\includegraphics[width=4cm]{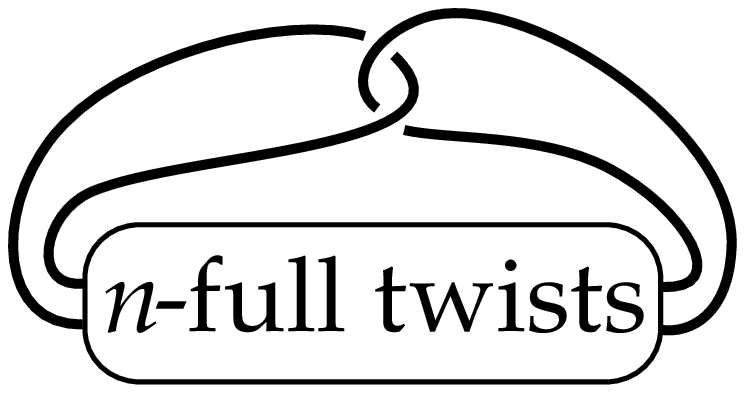}\\[6mm] 
\end{center} 
We have $J(2,0)=0_1$ (unknot), $J(2,2)=3_1$ (trefoil), $J(2,4)=5_2$, $J(2,-2)=4_1$ (figure-eight knot). 
If we regard a 1/2-full twist as a half twist, then $J(2, -2n)$ and $J(2,2n+1)$ are the mirror image of each other. 
The group $\pi=\pi_1(S^3-J(2,2n))$ of $J(2,2n)$ has the following standard presentation 
\cite[Proposition 1]{HosteShanahan2004JKTR}: 
\[\pi=\left<a,b \,|\, aw^n=w^nb\right>, \ \  	w=[a,b^{-1}]=ab^{-1}a^{-1}b.\]
By the theory of character varieties, conjugacy classes of ${\rm SL}_2$-representations are parametrized by $x=\tr \rho(a)$ and $y=\tr \rho(ab)$. We also put $z=z(x,y)=\tr \rho(w)=2x^2-x^2y+y^2-2$. 
Define the Chebyshev polynomials $\mca{S}_n(z)\in \Z[z]$ of the second kind by $\mca{S}_n(2\cos \theta)=\dfrac{\sin n\theta}{\sin \theta}$, $\theta\in \R$. (Note that in some other articles a slightly different convention $S_{n+1}=\mca{S}_n$ is used.)

\begin{prop}[{\cite[Theorem 3.3.1]{Le1993}, \cite[(1.2), (1.3)]{NagasatoTran2016OJM}, \cite[Propositions 3.5, 3.6]{TangeTranUeki2022IMRN}}] \label{prop.charvar} 
Each conjugacy class of irreducible ${\rm SL}_2\C$-representations corresponds to each zero of $f_n(x,y)=(y-1)\mca{S}_n(z)-\mca{S}_{n-1}(z)$ with $x^2-y-2\neq 0$. 
Each conjugacy class of non-abelian reducible $\SL_2\C$-representations corresponds to each zero of $x^2-y-2$. 

If we replace $\C$ by a finite field $\F$, then a similar assertion holds for absolutely irreducible/reducible representations. 
\end{prop} 


\begin{prop} \label{prop.reducible} 
{\rm (1)} The classical Alexander polynomial (corresponding to $\rho=1$) is $\Delta_0(t)=t-1$, $\Delta_1(t)=nt^2-(2n-1)t+n$.

{\rm (2)} 
Let $\rho:\pi\to {\rm SL}_2 O$ be a representation  with $x=\tr \rho(a)$ and suppose that $\rho\otimes \C$ is irreducible.  
Then 
$\Delta_{\rho,0}(t)\, \dot{=}\, t^2-xt+1$ or its divisor $t-\kappa$, and 
$\Delta_{\rho,1}(t)\, \dot{=}\, n^2t^4+n(2n-1)t^3+((2n-1)^2+n(2n-1)x)t^2+n(2n-1)t+n^2$ or this term times $1/(t-\kappa^{-1})$. 

The similar assertion holds for a representation $\ol{\rho}:\pi\to {\rm SL}_2\F$ over a finite field $\F$. 
\end{prop} 

\begin{proof} 
We prove (2). Since the Reidemeister torsion of an ${\rm SL}_2\C$-representation is defined only up to multiplication by $t^\Z$, we may work over $\C$. 
By a simultaneous upper triangularization, we may assume that 
a reducible representation $\rho:\pi\to {\rm SL}_2\C$ is given by $\rho(a)=\spmx{\kappa&1\\0&\kappa},$ $\rho(b)=\spmx{\kappa&\ast\\0&\kappa}$. 
By Kitano--Morifuji's argument in \cite[Proof of Theorem 3.1]{KitanoMorifuji2005}, 
the twisted Alexander polynomial is given by the classical Alexander polynomials as 
$\dfrac{\Delta_{\rho,1}(t)}{\Delta_{\rho,0}} \, \dot{=}\,\dfrac{\Delta_K(\kappa t)\Delta_K(\kappa^{-1} t)}{(\kappa t -1)(\kappa^{-1}t-1)}$. 
Since 
$\Delta_{\rho,0}$ is the gcd of the 2-minors of \[\spmx{\rho(a)t-I_2&\rho(b)t-I_2}=\spmx{\kappa t-1&1&\kappa t-1& \ast\\ 0&\kappa^{-1}t-1&0&\kappa^{-1}t-1},\] 
we see that $\Delta_{\rho,0}(t)$ is $t^2-xt+1$ or its divisor of degree 1. Together with (1), we obtain the assertion for $\rho:\pi\to {\rm SL}_2O$. 
If we instead work over an algebraic closure $\ol{\F}_p$ of $\F$, 
we similarly obtain the assertion for $\ol{\rho}:\pi\to {\rm SL}_2\F$ over a finite field $\F$. 
\end{proof} 

\begin{prop}[{\cite[Theorem 1]{Tran2018KMJ}}] \label{prop.Tran} 
Let $\rho:\pi\to \SL_2 O$ be a representation with $x=\tr \rho(a)$, $y=\tr \rho(ab)$. 
If $\rho\otimes \C$ is irreducible, then we have 
\begin{gather*}
\Delta_{\rho,1}(t)/\Delta_{\rho,0}(t)=a_0(x,y)t^2+a_1(x,y)t+a_0(x,y),\\
a_0(x,y)=\dfrac{\mca{S}_{n+1}(z)-\mca{S}_{n-1}(z)-2}{z-2}, \ 
a_1(x,y)=x(\mca{S}_{n}(z)-a_0(x,y)),
\end{gather*} 
 where $a_0(x,y), a_1(x,y)\in \Z[x,y]$.

For an absolutely irreducible representation $\ol{\rho}:\pi\to \SL_2 \F$ over a finite field $\F$,
the similar assertion with $a_0(x,y), a_1(x,y)\in \F_p[x,y]$ holds.
\end{prop}

\begin{prop}[{\cite[Theorem A, Proposition 3.5]{TangeTranUeki2022IMRN}}] \label{prop.non-acyclic} 
If $(x,y) \in \C^2$ or $\F^2$, then the condition  
\[\Delta_{\rho,1}(1)/\Delta_{\rho,0}(1)=2a_0+a_1=0,\ f_n=0,\ x^2-y-2\neq 0\] 
is equivalent to that
 \[x=y, \ 1-x=\alpha+\alpha^{-1}\] 
for some $3n-1$-th root $\alpha$ of unity with $\alpha\neq \pm1$. 

Each point $(x,y)$ satisfying these conditions corresponds to each conjugacy class of non-acyclic irreducible ${\rm SL}_2$-representations over $\C$ or $\F$. 
\end{prop} 

The following is a key to prove our $\mu=0$ theorem. 

\begin{prop} \label{prop.lem} 
For an absolutely irreducible ${\rm SL}_2$-representation $\rho$ over $\C$ or a finite field $\F$, we have 
$\tau_{\rho\otimes \alpha}(t)=\Delta_{\rho,1}(t)/\Delta_{\rho,0}(t)\neq 0$. 
\end{prop}

\begin{proof}
We suppose that $a_0$, $a_1$, $f_n$ have a common zero $(x,y)$ to deduce contradiction; 
Since $2a_0+a_1=0$ and $f_n=0$, $a_0=0$ and $a_1=0$ implies $x\mca{S}_n(z)=0$. 
Putting $1-x=\alpha+\alpha^{-1}$ ass in Proposition \ref{prop.non-acyclic}, 
by $\mca{S}_n(z)=\dfrac{\alpha^{3n}-\alpha^{-3n}}{\alpha^3-\alpha^{-3}}$, we obtain $\alpha^{6n}=1$. 
On the other hand, by Proposition \ref{prop.non-acyclic}, $2a_0+a_1=0$ and $f_n=0$ together with $x^2-y-2\neq 0$ yield $\alpha^{3n-1}=1$ and $\alpha\neq \pm1$. 
Hence contradiction. Thus we have $\Delta_{\rho}(t)\neq 0$. 
\end{proof}

\begin{rem} All arguments and assertions above persist if $\rho$ and $\C$ are replaced by $\rho_p:\pi\to {\rm SL}_2 O_\p$ and $\C_p$. 
\end{rem} 

Combining all above, we obtain the following: 

\begin{thm} \label{thm.twistknot} 
For any representation $\rho:\pi\to {\rm SL}_2O$ or $\rho_{\p}:\pi\to {\rm SL}_2 O_\p$ and any $\Z/m\Z\times \Zp$-cover of any twist knot $J(2,2n)$, the equality $\mu_{p,i}=0$ for $i=0,1,\tau$ holds. 
\end{thm}

\begin{proof} 
If the residual representation $\rho \mod \p$ is absolutely irreducible, then by \cite[Proposition A]{FriedlKimKitayama2012} and \cite[Corollary 3]{TangeRyoto2018JKTR} we see that 
$\Delta_{\rho,0}(t)$ $\dot{=}$ $1$ in $O_\p[t^\Z]$, and hence $\mu_{p,0}=0$ holds. 
By Proposition \ref{prop.lem} (1), we have $\Delta_{\ol{\rho},1}\neq 0$. 
Hence by Theorem \ref{thm.mu=0}, we obtain $\mu_{p,1}=\mu_{p,\tau}=0$. 

If $\ol{\rho} =\rho \mod \p$ is absolutely reducible, then by $x=\tr \rho(a)\in O$ or $O_\p$, 
Proposition \ref{prop.reducible} yields that $\Delta_{\rho, i}(t) \in O[t^\Z]$ is monic. 
Indeed, consider the divisor \[\delta(t)=n^2t^4+n(2n-1)t^3+((2n-1)^2+n(2n-1)x)t^2+n(2n-1)t+n^2.\]
If $p {\not{|}}\, n$, then $\delta(t)$ is monic. 
If $p|n$, then $\delta(t)=2n-1-nx=-1\neq 0$ is again monic. 
Since other divisors are all monic, we obtain $\mu_{p,i}=0$. 
\end{proof}

For a knot group, the coefficients of the twisted Alexander polynomials may be seen as continuous functions on the character variety. In the cases of twist knots, $\ol{\rho}$ with $\Delta_{\ol{\rho},1}=0$ corresponds to the common zeros of three polynomials $f_n(x,y)$, $a_0(x,y)$, $a_1(x,y)$ modulo $\p$ in two variables. 
We have similar situations for higher representations or for other knots, so for most cases, we may expect $\mu=0$ and a case with $\mu>0$ should be very rare. 

Tran's result \cite[Theorem 1]{Tran2018KMJ} that we cited for twist knots in Proposition \ref{prop.Tran} is originally stated for general two-bridge knots with genus 1. 
Although his assertion is stated for representations over $\C$, the argument makes sense over $\Z$, 
and hence it is applicable to the case over any field. 
Systematic exploration of non-acyclic representation for twist knots was given by \cite{Benard2020OJM, TangeTranUeki2022IMRN} and that for other knots and links are ongoing \cite{BenardTangeTranUeki-Whitehead}.

\section{Residually reducible irreducible representations} 
This section is optional. 
We investigate residually reducible irreducible ${\rm SL}_2$-representations of twist knots to give a partial alternative proof of Theorem \ref{thm.twistknot} stating $\mu=0$.

\begin{prop} \label{prop.lem2} 
If $x,y$ belong to a field, 
then the intersection of $f_n(x,y)=0$ and $x^2-y-2=0$ is given by $(\pm\sqrt{4-\frac{1}{n}},2-\frac{1}{n})$, if exist. 
\end{prop}

\begin{proof}
By $z-2=(y-2)(x^2-y-2)=0$, we have $z=2$, hence 
\[f_n(x,y)=(y-1)\mca{S}_n(z)-\mca{S}_{n-1}(z)=(y-1)n-(n-1)=yn-2n+1=0.\] 
If $n\neq 0$, then we have $y=2-\frac{1}{n}$, and $x=\pm\sqrt{y+2}=\pm\sqrt{4-\frac{1}{n}}$. If $n=0$, then the intersection is empty. (Note that this may happen if the field is of characteristic $p>0$ and $p|n$.) 
\end{proof} 

\begin{thm} \label{thm.resred} 
If $p|(3n-1)$, then the point $(\pm1,-1)$ in $\F_p\!^2$ corresponds to conjugacy classes of reducible representations on the curve $f_n=0$. 
Each lift of these points in $O^2$ corresponds to a conjugacy class of representations $\rho:\pi\to {\rm SL}_2 O$ such that $\rho\mod \p$ is reducible and $\rho\otimes \C$ is irreducible. This implies there are infinitely many such conjugacy classes. 
For such a $\rho$ and any $\Z/m\Z\times \Zp$-cover, $\mu_{p,i}=0$ for $i=0,1,\tau$ holds.
\end{thm}

\begin{proof}
By Proposition \ref{prop.lem} and \cite[Proposition 1.5]{TangeTranUeki2022IMRN}, 
if $(x,y)=(\pm1,-1)$, then we have $z(x,y)=2$, $\mca{S}_n(z)=n$, 
hence $f_n(x,y)=-3n+1$. 
By $\tau_{\rho(x,y)\otimes \alpha}(1)=\tau_n(x,y)=a_0+2a_1=(2-x)n^2+xn,$ we have \[\tau_n(1,-1)=n^2+n, \ \ \tau_n(-1,-1)=-3n^2+n\] in $\F_p$. 
We in addition have $x^2-y-2=0$. 
If $p|(3n-1)$, then $(-1,-1)\in \F_p\!^2$ corresponds to the conjugacy class of reducible non-acyclic representations $\ol{\rho}:\pi\to {\rm SL}_2 \F_p$. 
If in addition $p=2$ and $n$ is odd, then the similar holds for $(1,-1)\in \F_p\!^2$. 
(Lifts of $\ol{\rho}$ are candidates for $\mu>0$.) 
Now by the proof of Proposition \ref{prop.lem}, 
if we suppose $a_0(x,y)=a_1(x,y)=0$, then we have $x\mca{S}_n(z)=\pm n=0$, contradicting $p|(3n-1)$. 
Hence we have \[\Delta_{\ol{\rho},1}(t)/\Delta_{\ol{\rho},0}(t) = a_0(x,y) t^2+ a_1(x,y)t +a_0(x,y) \neq 0\] in $\F[t^\Z]$. 

A point $(\xi, \eta) \in O^2$ with its image being $(\pm1,-1)$ in $\F_\p\!^2$ is called a lift of $(\pm1,-1)$. 
There are infinitely many lifts if we may change $O$. 
For instance, for each $k\in \Z$, put $\eta=pk-1$, and let $x=\xi$ be a root of $f_n(x,\eta)$ of in the decomposition field $F$. By Proposition \ref{prop.charvar} and \ref{prop.lem2}, all but at most two points, if $\rho$ corresponds to $(\xi,\eta)$, then $\rho\otimes \C$ is irreducible, and $\Delta_{\rho,1}(t)\neq 0$ holds. 
Taking a quadratic extension and a conjugate if needed, we have $\ol{\rho}=\rho \mod \p$, hence $\Delta_{\rho,i}(t)\mod \p =\Delta_{\ol{\rho},i}(t)$.  
By $\Delta_{\ol{\rho},1}(t)/\Delta_{\ol{\rho},0}(t)\neq 0$, we obtain $\mu_{p, \tau}=0$. 
A direct calculation proves $\Delta_{\rho,0}(t)\neq 0$ in $\F_p[t^\Z]$, hence $\mu_{p,0}=0$. 
Therefore we obtain $\mu_{p,1}=0$. 
\end{proof} 

\begin{rem} 
The point $(-1,-1) \in \F^2$ is a singular point of $f_n$. 
Indeed, by the calculation of the partial derivatives in \cite[Sections 1.6, 3.2]{TangeTranUeki2022IMRN}, we have
\[\dfrac{\partial f_n}{\partial x} (\xi,\xi)=\dfrac{2(n\xi^2-2n\xi-1)}{(\xi+1)\xi(\xi-2)(\xi-3)}, \ 
\dfrac{\partial f_n}{\partial y} (\xi,\xi)=\dfrac{2((2n-1)\xi^2+(-4n+2)\xi+1)}{(\xi+1)\xi(\xi-2)(\xi-3)}\]
 at $\xi\neq -1,0,2,3$.  
By L'H\^{o}pital's rule, we have $\dfrac{\partial f_n}{\partial x} (-1,-1) = \dfrac{\partial f_n}{\partial y} (-1,-1)= 0$ in $\F$. 
This means that we cannot apply Hensel's lemma to obtain $(\xi,\eta)$ in the proof above. 
\end{rem}

\section{Remarks} 

Here, we paraphrase famous conjectures related to twisted Alexander polynomials into twisted Iwasawa invariants and attach remarks. 


\begin{conj}[{\cite[Problem 1.12 (J.~Simon)]{Kirby1997problems}}] 
If there is a surjective homomorphism $\varphi:\pi_J\surj \pi_K$ between the group of knots $J$ and $K$, 
then their genera satisfy $g(J)\geq g(K)$. 
\end{conj}

There are several studies on surjective homomorphisms (epimorphisms) of knot groups (cf. \cite{KitanoSuzukiWada2005AGT, OhtsukiRileySakuma2008GT}). 
If a representation $\rho: \pi_K\to \GL_2O$ and a surjective homomorphism $\varphi:\pi_J\surj \pi_K$ is given, 
then we have $\Delta_{J,\rho\circ \varphi}(t) \mid \Delta_{K,\rho}(t)$, and hence their Iwasawa $\lambda$-invariants satisfy $\lambda_J| \lambda_K$. 
In order to prove the conjecture in the affirmative, it suffices to show the existence of a representation $\rho$ detecting $g(K)$ such that the composition $\rho\circ \varphi$ detects $g(J)$ as well. 
It would be interesting to 
approach this problem by paraphrasing Friedl--Vidussi's argument \cite{FriedlVidussi2011AnnMath, FriedlVidussi2015Crelle} as well as Friedl--L\"{u}ck's succeeding study \cite{FriedlLuck2019PLMS} 
 into our language.  


\begin{conj}[\cite{DunfieldFriedlJackson2012}] 
Let $K$ be a hyperbolic knot with genus $g(K)$ and let $d$ denote the degree of the twisted Alexander polynomial $\mca{T}_K$ of the holonomy representation. Then we have $d=4g(K)-2$. 
\end{conj} 
We observed in Example \ref{eg.holonomy} that $\lambda$'s distinguish the spin structures of the figure-eight knot. 
It would be interesting to investigate what information the Iwasawa invariants of the holonomy representations convey in general. 

The Reidemeister torsions of the symmetric powers of the holonomy representation converge to the hyperbolic volume \cite{Goda2017pja, BenardDuboisHeusenerPorti2022}. 
If one could formulate its $p$-adic analogue, for instance by using Kionke--Loeh's $p$-adic simplicial volume \cite{KionkeLoeh2021Glasgow}, then it should be related to our work. 




\begin{rem}[(Yi Liu)] For an arithmetic manifold, we may extend our study to non-cocompact representations. 
\end{rem} 
We have so far dealt with $\p$-adic representations that are not necessarily obtained from representations over $O$ as well, 
which exceedingly played essential roles in the proof of Theorem \ref{thm.profinite} on profinite rigidity. 
Studying specific $\p$-adic representations would be of further interest (cf.\cite{BridsonMcReynoldsReidSpitler2020}, \cite{YiLiu2023Peking}).

Finally, we would like to remark that our Conjecture \ref{conj.mu} on $\mu$-invariants would suggest a new, slightly deep interaction between number theory and low dimensional topology, as arithmetic topology has aimed. 
The unsolved part is the existence of a lift $\rho$ of a higher dimensional non-acyclic representation $\ol{\rho}$ with $\Delta_{\rho}(t)\neq 0$. 
Dohyeong Kim recently pointed out in his talks \cite{DKim2023talkLDTNT} that $\mu>0$ of non-fibered knots might come from a certain \emph{local condition} in the sense of Hida theory and that should suggest a refinement of the analogue of the Selmer group, an important algebraic object that has been studied in the series of our works \cite{MTTU2017, KMTT2018, TangeTranUeki2022IMRN, BenardTangeTranUeki-Whitehead}. 
We expect that this viewpoint could also give clues to Morifuji--Suzuki's project \cite{MorifujiSuzuki2022IJM} on finding representations with vanishing Alexander polynomials, and to our conjecture as well. 

\section*{Acknowledgments} 
We would like to express our sincere gratitude to 
Takashi Hara, 
Tetsuya Ito, 
Teruhisa Kadokami, 
Dohyeong Kim, 
Takahiro Kitayama, 
Yi Liu, 
Yasushi Mizusawa, 
Takayuki Morifuji, 
Masanori Morishita, 
Tatsuya Ohshita, 
Makoto Sakuma, 
Masaaki Suzuki, and anonymous referees  
for useful comments and information. 
The second author has been partially supported by JSPS KAKENHI Grant Numbers JP19K14538 and JP23K12969.

\bibliographystyle{amsalpha}
\bibliography{TangeUeki_twistedIwasawa.arXiv4.bbl}

\end{document}